\documentclass[
final
]{dmtcs-episciences}

\usepackage[utf8]{inputenc}
\usepackage{subfigure}

\usepackage[round]{natbib}

\usepackage{amssymb,amsmath,amsthm,latexsym,bbm,microtype, mathtools, varwidth}

\theoremstyle{plain}
\newtheorem{thm}{Theorem}

\newtheorem{cor}{Corollary}

\newtheorem{prop}{Proposition}

\theoremstyle{definition}
\newtheorem{exam}{Example}
\newtheorem*{defi}{Definition}

\theoremstyle{remark}

\newcommand\RR{\mathbb{R}}
\newcommand\Def[1]{{\bf #1}}

\author{Matthias Beck\affiliationmark{1}
  \and Sampada Kolhatkar\affiliationmark{2}}
\title[Bivariate Chromatic Polynomials of Mixed Graphs]{Bivariate Chromatic Polynomials of Mixed Graphs}
\affiliation{
Department of Mathematics, San Francisco State University, U.S.A.\\
 Institut f\"ur Mathematik, Freie Universit\"at Berlin, Germany}
\begin{document}
\publicationdata{vol. 25:2 }
{2023}
{2}
{10.46298/dmtcs.9595}
{2022-05-23; 2022-05-23; 2022-11-03}{2023-06-01}

\maketitle
\begin{abstract}
The bivariate chromatic polynomial $\chi_G(x,y)$ of a graph $G = (V, E)$,
introduced by Dohmen--P\"{o}nitz--Tittmann (2003), counts all $x$-colorings of $G$ such that
adjacent vertices get different colors if they are $\le y$. We extend this notion to mixed
graphs, which have both directed and undirected edges. Our main result is a decomposition
formula which expresses $\chi_G(x,y)$ as a sum of bivariate order polynomials
(Beck--Farahmand--Karunaratne--Zuniga Ruiz 2020), and a combinatorial reciprocity theorem
for $\chi_G(x,y)$.
\end{abstract}

\section{Introduction}
\label{sec:in}

Graph coloring problems are ubiquitous in many areas within and outside of mathematics. Our interest is in enumerating proper colorings for graphs, directed graphs, and mixed graphs (and in the latter two instances, there are two definitions of the notion of a coloring being proper).

The motivation of our study is the \Def{bivariate chromatic polynomial} $\chi_G(x,y)$
of a graph $G = (V, E)$, first introduced in \cite{dohmenponitztittmann} and
defined as the counting function of colorings $c : V \to [x] := \left\{ 1, 2, \dots, x \right\}$ that satisfy for any edge $vw \in E$
\[
  c(v) \ne c(w) \quad \text{ or } \quad c(v) = c(w) > y \, .
\]
The usual univariate chromatic polynomial of $G$ can be recovered as the special evaluation $\chi_G(x,x)$.
Dohmen, P\"onitz, and Tittmann provided basic properties of $\chi_G(x,y)$
in~\cite{dohmenponitztittmann}, including polynomiality and special evaluations which yield the matching and independence polynomials of $G$.
Subsequent results include a deletion--contraction formula and applications to Fibonacci-sequence identities~\cite{hillarwindfeldt}, common generalizations of $\chi_G(x,y)$ and the Tutte polynomial~\cite{averbouchgodlinmakowsky}, and closed formulas for paths and cycles~\cite{dohmenbivariatepathsandcycles}.

We initiate the study of a directed/mixed version of this bivariate chromatic polynomial. Since directed graphs form a subset of mixed graphs, we may restrict our definitions to a \Def{mixed graph} $G = (V, E, A)$ consisting, as usual, of a set $V$ of vertices, a set $E$ of (undirected) edges, and a set $A$ of arcs (directed edges).
Coloring problems in mixed graphs have various applications, for example in scheduling problems in which one has both disjunctive and precedence constraints (see, e.g., \cite{furmanczykkosowkiries,HaKuWe,SoTaWe}).

\begin{defi}For a mixed graph $G = \left(V, E, A\right)$, the \Def{bivariate
chromatic polynomial} $\chi_G(x,y)$, where $1 \leq y \leq x$, is defined as  the counting function of colorings $c : V \longrightarrow \left[x\right]$ that satisfies
for every edge $uv \in E$
\begin{align}
c(u) \neq c(v) \quad &\text{ or } \quad c(u)>y \label{cond:bcpomg1}
\end{align}
and for every arc $\overrightarrow{uv} \in A$
\begin{align}
c(u) < c(v) \quad &\text{ or } \quad c(u) > y \, . \label{cond:bcpomg2}
\end{align}
\end{defi}
It is not obvious that this counting function is a polynomial in $x$ and $y$; we will prove
this as a by-product of Theorem~\ref{mainthm2} below. 
Naturally, for a mixed graph with $A = \emptyset$, we recover the Dohmen--P\"{o}nitz--Tittmann chromatic polynomial above.
On the other hand, $\chi_G(x,x)$ is the univariate chromatic polynomial of the mixed graph
$G$;
\cite{sotskovtanaev}  showed that this function (if not identically zero) is indeed a polynomial in $x$ of degree $|V|$ and computed the two leading coefficients.
We note that $\chi_G(x,x)$ is sometimes called the \emph{strong} chromatic polynomial of
$G$, because there is an alternative notion of a proper coloring of $G$ in which the
inequality in~\eqref{cond:bcpomg2} is replaced by $\le$ (see, e.g.,~\cite{kotekmakowskyzilber}).

\begin{exam} Consider the directed graph  with  $V = \left\{ u,v\right\},$ $E = \emptyset$,  $A = \left\{\overrightarrow{uv}\right\}$. A quick case analysis yields the bivariate chromatic polynomial 
\begin{align*}
\chi_{G}(x,y) &= {x \choose 2} + y(x-y) + {x-y+1 \choose 2} = \frac{1}{2} \left( 2x^2 - y^2 -y \right).
\end{align*}
\end{exam}

\begin{figure}
\centering
\includegraphics[width=0.25\textwidth]{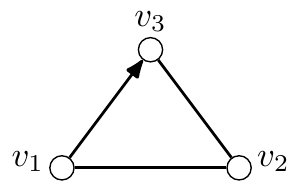}
\caption{A mixed graph $G$.}
\label{fig:eg2fig1}
\end{figure}
\begin{exam}
The mixed graph in Figure~\ref{fig:eg2fig1} has bivariate chromatic polynomial
\[
\chi_G(x,y) = x^3 - \frac{1}{2}xy^2 - \frac{5}{2}xy+y^2+y
\]
as we will compute in Section~\ref{sec:decomp}.
\end{exam}

After providing some background in Section~\ref{sec:background}, we derive in
Section~\ref{sec:delcontr} deletion--contrac\-tion formulas for $\chi_{G}(x,y)$. Our main
results are in Section~\ref{sec:decomp}, where we decompose $\chi_{G}(x,y)$ into bivariate
order polynomials (originally introduced in~\cite{Beckbop} and loosely connected with the marked poset concepts of~\cite{ardilabliemsalazar}), and Section~\ref{sec:bcpmgReciprocity}, where we give a combinatorial reciprocity theorem interpreting $\chi_{G}(-x,-y)$.
Our results recover known theorems for undirected graphs (the case $A =
\emptyset$)~\cite{Beckbop}. Bivariate order polynomials are the natural counterparts of bivariate chromatic polynomials in the theory of posets. Our work reveals that bivariate order polynomials are as helpful in the setting of mixed graphs as they are for undirected graphs.

\section{Chromatic and (Bicolored) Order Polynomials}\label{sec:background}

For a finite poset $(P, \preceq)$, 
\cite{stanleychromaticlike} (see
also~\cite[Chapter 3]{stanleyec1}) famously introduced the ``chromatic-like'' \Def{order
polynomial} $\Omega_P(x)$ counting all \Def{order preserving maps} $\varphi : P \to [x]$, that is, 
\[
  a \preceq b \qquad \Longrightarrow \qquad \varphi(a) \le \varphi(b) \, .
\]
Here we think of $[x]$ as a chain with $x$ elements, and so $\le$ denotes the usual order in $\RR$.
The connection to chromatic polynomials is best exhibited through a variant of
$\Omega_P(x)$, namely the number $\Omega_P^\circ(x)$ of all \Def{strictly order preserving
maps} $\varphi : P \to [x]$:
\[
  a \prec b \qquad \Longrightarrow \qquad \varphi(a) < \varphi(b) \, .
\]
When thinking of $P$ as an acyclic directed graph, it is a short step interpreting $\Omega_P^\circ(x)$ as a directed version of the chromatic polynomial.
Along the same lines, one can write the chromatic polynomial of a given graph $G$ as
\begin{equation}\label{eq:chiintoomegas}
  \chi_G(x) \ = \sum_{ \sigma \text{ acyclic orientation of } G } \Omega_\sigma^\circ(x) \, .
\end{equation}
Stanley's two main initial results on order polynomials were
\begin{itemize}
\item decomposition formulas for $\Omega_P(x)$ and $\Omega_P^\circ(x)$ in terms of certain permutation statistics for linear extensions of~$P$, from which polynomiality of
$\Omega_P(x)$ and $\Omega_P^\circ(x)$ also follows;
\item the combinatorial reciprocity theorem
$
  (-1)^{ |P| } \, \Omega_P(-x) \ = \ \Omega_P^\circ(x) \, .
$
\end{itemize}
The latter, combined with~\eqref{eq:chiintoomegas}, gives in turn rise to
\begin{itemize}
\item Stanley's reciprocity theorem for chromatic polynomials: $(-1)^{ |V| } \, \chi_G(-x)$ equals the number of pairs of an $x$-coloring and a compatible acyclic
orientation~\cite{stanleyacyclic}.
\end{itemize}
Reciprocity theorems for the two versions of univariate chromatic polynomials of mixed
graphs were proved in~\cite{Golomb,weakmixed}.

It is natural to extend order polynomials and the three bullet points above to
the bivariate chromatic setting, and this was done for (undirected) graphs in~\cite{Beckbop}.
As we will need it below, we recall the setup here.
The finite poset  $\left(P, \preceq\right)$ is called a  \Def{bicolored poset} if $P$ can be viewed as the disjoint union of sets $C$ and $S$, 
whose elements are called \Def{celeste} and \Def{silver}, respectively.
This color labeling of the elements of the bicolored poset is captured in the order
preserving maps by introducing another variable, as follows. 
A map $\varphi : P \longrightarrow \left[x\right]$ is called an \Def{order preserving $(x,y)$-map} if
\[
  a \preceq b \ \Longrightarrow \ \varphi(a) \le \varphi(b) \ \ \text{ for all } a, b \in P
  \qquad \text{ and } \qquad
  \varphi(c) \ge y \ \ \text{ for all } c \in C \, .
\]
The function $\Omega_{ P, \, C } (x, y)$ counts the number of order preserving $(x, y)$-maps.
The map $\varphi : P \longrightarrow \left[x\right]$ is a \Def{strictly order preserving $(x,y)$-map} if
\[
  a \prec b \ \Longrightarrow \ \varphi(a) < \varphi(b) \ \ \text{ for all } a, b \in P
  \qquad \text{ and } \qquad
  \varphi(c) > y \ \ \text{ for all } c \in C \, .
\]
The function $\Omega^{\circ}_{ P, \, C } (x, y)$ counts the number of strictly  order preserving $(x, y)$-maps.
The functions $\Omega^{\circ}_{ P, \, C } (x, y)$ and  $\Omega_{ P, \, C } (x, y)$ are called \Def{bivariate order polynomial} and \Def{weak bivariate order polynomial}, respectively.
They are indeed polynomials, which can be computed via certain descent statistics, and which
are related via the combinatorial reciprocity~\cite{Beckbop}
\begin{equation}\label{eq:reciprocityBOP}
  (-1)^{ |P| } \, \Omega^{\circ}_{ P, \, C } (-x,-y) \ = \ \Omega_{ P, \, C } (x,y+1) \, .
\end{equation}

As we mentioned in the introduction, bivariate order polynomials exhibit a connection to the theory of marked posets introduced in \cite{ardilabliemsalazar}. Briefly, one marks here the celeste elements, with a lower bound of $y$, and demands the lower bound 0 and the upper bound $x$ throughout the poset.


\section{Deletion--Contraction}\label{sec:delcontr}
We start developing the properties of $\chi_G(x,y)$ by providing deletion--contraction
formulas.
For a mixed graph $G= \left(V,E,A\right)$, let $G-e$ denote edge deletion and $G/e$ denote edge contraction for an edge $e$ of $G$; let $v_e$ denote the vertex obtained after the contraction of  edge $e$. We use a similar terminology for deleting/contracting an arc. For an arc $a$ of mixed graph $G$, let $G_a \coloneqq \left(V,E, A-\left\{\overrightarrow{uv} \right\} \cup \left\{  \overrightarrow{vu} \right\}
\right)$, that is, $G_a$ is the graph obtained by reversing the direction of arc~$a$.
\begin{prop} \label{prop:delcontredgebcpmg}
If $G= \left(V,E,A\right)$ is a mixed graph and $e \in E$ is an edge, then 
\begin{equation}\label{eq:delcontredge}
\chi_G(x,y) \ = \ \chi_{G-e}(x,y) - \chi_{G/e}(x,y) + (x-y) \chi_{(G/e) -v_e}(x,y) \, .
\end{equation}
If $a = \overrightarrow{uv} \in A$ is an arc, then
\begin{equation}\label{eq:delcontrarc}
\begin{split}
\chi_G(x,y) + \chi_{G_a}(x,y) \ &= \ \chi_{G-a} (x,y) - \chi_{G/a} (x,y) + (x-y)(1-x+y)\chi_{(G/a) -v_a}(x,y)  \\
& \hspace{15pt} + (x-y) \left( \chi_{G-a-v}(x,y) + \chi_{G-a-u}(x,y) \right) . 
\end{split}
\end{equation}
\end{prop}

We remark that \eqref{eq:delcontredge} is equivalent to
\cite[Proposition~1]{averbouchgodlinmakowskyelimination}.



\begin{proof}[of~\eqref{eq:delcontrarc}] 
Given $a =\overrightarrow{uv} \in A$, let $C$ be the set of all bivariate colorings of $G$ and $C_a$ the set of all bivariate colorings of $G_a$.
By inclusion--exclusion, 
\begin{align}
\chi_G(x,y) + \chi_{G_a}(x,y) & = \lvert C \cup C_a \rvert  + \lvert C \cap C_a \rvert. \label{eq:G+Ga}
\end{align}

For a coloring $c \in C \cup C_a$, we count the number of ways the following coloring conditions are satisfied: $c(u) < c(v)$ or $c(v) < c(u)$ or $c(u)> y$ or $c(v) > y$. This means, we have to count the number of ways of coloring vertices $u$ and $v$ such that they can have any color labels from the set $\left\{1,2,\ldots,x\right\}$ except that the vertices can not have equal colors with labels in the set $\left\{1,2,\ldots,y\right\}$. This is exactly counted by 
\begin{align} \chi_{G-a} (x,y) - \chi_{G/a} (x,y) + (x-y) \, \chi_{(G/a)-v_a} (x,y) \ = \
\lvert C \cup C_a \rvert \, . \label{eq:CunionCa}
\end{align}

\noindent For a coloring $c \in C \cap C_a$ we distinguish between the following
cases.. 
\begin{itemize}
\item[]Case 1: $c(u) < c(v)$ and $c(u) > c(v) $.
\newline There does not exist a feasible coloring in $C \cap C_a$ that satisfies these conditions simultaneously. 
\item[]Case 2: $ y <  c(v)$ with $ c(u) \leq c(v)$ and  $ y  < c(u)$ with $ c(v) \leq
c(u)$.
\newline This implies the coloring condition $y < c(u) = c(v)$, which is counted in
$(x-y) \, \chi_{(G/a)-v_a}(x,y) $ ways.

\item[]Case 3:  $c(u) < c(v)$  and $ y <  c(v)$ with $ c(u) \leq  c(v)$
\newline This implies that the coloring $c$ must satisfy  $ y <  c(v)$ with $ c(u) < c(v)$. There are two possibilities:
\begin{itemize}
\item[{$\bullet$}] $y <c(u) < c(v) \leq x$.

The colors for $u$ and $v$ can be chosen in ${x-y \choose 2}$ ways.  Thus the
number of possible colorings is ${x-y \choose 2}\chi_{(G/a) -v_a}(x,y)$.

\item[{$\bullet$}] $1 \leq c(u) \leq y < c(v) \leq x$.

There are $(x-y)$ ways to color $v$. To color $u$, the condition $1 \leq c(u) \leq y$
needs to be satisfied. This is equivalent to counting colorings where $c(u) \leq x$
and removing the possible colorings with $c(u) > y$, giving $(x-y) \left(  \chi_{G-a-u}(x,y) - (x-y) \chi_{(G/a) -v_a}(x,y)\right)$ colorings .
\end{itemize}
In total there are ${x-y \choose 2}\chi_{(G/a) -v_a}(x,y) + (x-y) \big(
\chi_{G-a-u}(x,y) - (x-y) \, \chi_{(G/a) -v_a}(x,y)\big)$ colorings. 

\item[]Case 4:   $c(v) < c(u)$  and $ y <  c(u)$ with $ c(v) \leq c(u)$
\newline This implies that the coloring $c$ must satisfy  $ y <  c(u)$ with $ c(v) < c(u)$. There are two possibilities:
\begin{itemize}
\item[{$\bullet$}] $y <c(v) < c(u) \leq x$.

The colors for $u$ and $v$ can be chosen in ${x-y \choose 2}$ ways.  Thus the possible colorings are counted by ${x-y \choose 2}\chi_{(G/a) -v_a}(x,y)$.

\item[{$\bullet$}] $1 \leq c(v) \leq y < c(u) \leq x$.

There are $(x-y)$ ways to color $u$. For coloring $v$, the condition $1 \leq c(v) \leq
y$ needs to be satisfied. This is equivalent to counting colorings where $c(v) \leq x$
and removing the possible colorings with $c(v) > y$, yielding $(x-y) \left(
\chi_{G-a-v}(x,y) - (x-y) \chi_{(G/a) -v_a} (x,y)\right)$ colorings.
\end{itemize}
In total there are \[{x-y \choose 2}\chi_{(G/a) -v_a}(x,y) + (x-y) \big(
\chi_{G-a-v}(x,y) - (x-y) \, \chi_{(G/a) -v_a}(x,y)\big)\] colorings. 
\end{itemize}

Thus
\begin{align}
 \vert C \cap C_a \vert \ &= \ (x-y) \, \chi_{(G/a) -v_a}(x,y)+ 2  {x-y \choose 2}   \chi_{(G/a) - v_a} (x,y) \nonumber \\
&\qquad + (x-y) \left[  \chi_{G-a-v}(x,y) - (x-y) \chi_{G/a -v}(x,y)\right] \nonumber \\ &\qquad + (x-y) \left[  \chi_{G-a-u}(x,y) -(x-y) \chi_{G/a -u}(x,y)\right]. \label{eq:CintersectionCa}
\end{align}
%

From Equations \eqref{eq:G+Ga},  \eqref{eq:CunionCa} and  \eqref{eq:CintersectionCa}
we finally obtain
\begin{align}
\chi_G(x,y) + \chi_{G_a}(x,y) \ 
& = \ \chi_{G-a} (x,y) - \chi_{G/a} (x,y) + (x-y)(1-x+y)\chi_{(G/a) -v_a}(x,y) \nonumber \\
& \hspace{15pt} + (x-y) \left[  \chi_{G-a-v}(x,y) +  \chi_{G-a-u}(x,y) \right].
\nonumber \qedhere
\end{align}
\end{proof}
%


\section{Decomposition into Order Polynomials}\label{sec:decomp}

 For a mixed graph $G= \left(V,E,A\right)$, we recall that a \Def{flat} of $G$ is a
mixed graph $H$ that can be constructed from $G$ by a series of contractions of edges
and arcs. We denote the sets of vertices, edges and arcs of the flat $H$ by $V(H), E(H)$ and
$ A(H)$, respectively. The subset of vertices of $H$ that results from contractions of
$G$ is denoted by $C(H)$.  An example is depicted in Figure~\ref{fig:3graphs}, where we
obtain the flat $H$ by contracting the edge $v_1v_4$. For this flat, the set of
contracted vertices is $C(H) = \{v_1v_4\}$.

For a mixed graph $G$, let $G^{u}$ denote the underlying undirected graph, that is, the graph obtained from $G$ by replacing its arcs with undirected edges. 
For some acyclic orientation $\sigma$ of $G^{u}$, let $T(\sigma)$ be the set of all tail vertices of arcs $a$ of a flat $H$ of $G$ for which the orientation of an edge in $\sigma$ is opposite to the direction of~$a$.

\begin{figure}[ht]
    \centering
    \subfigure[A mixed graph $G$.]{\includegraphics[scale=.8]{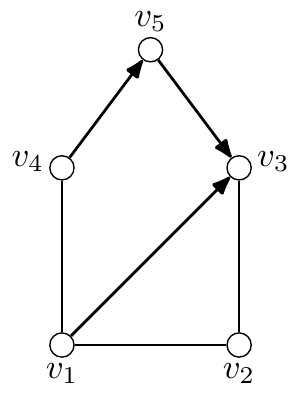}%
\label{fig:fig4a}}
    \qquad
    \subfigure[The flat $H$ obtained by contracting the edge $v_1v_4$.]{\includegraphics[scale=.8]{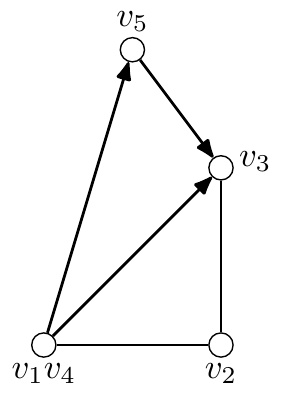}%
\label{fig:fig4b}}
    \qquad
    \subfigure[The underlying undirected graph.]{\includegraphics[scale=.8]{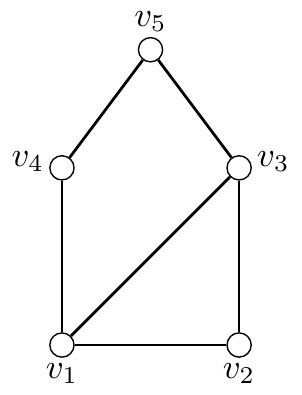}%
\label{fig:fig4c}}
    \caption{A mixed graph, one of its flat and the associated undirected graph.}\label{fig:3graphs}%
\end{figure}

\begin{thm} \label{mainthm2}
For a mixed graph $G$, 
\[
\chi_G(x,y) \ = \sum_{H \text{ \rmfamily flat of } G}
 \sum_{ \substack{ \sigma \text{ \rmfamily acyclic} \\ \text{\rmfamily orientation of } H^{u} } }\!\!\!\!
\Omega^{\circ}_{\sigma, \, C(H) \cup T(\sigma)} (x,y) \, . 
\]
\end{thm}

Note that this implies that $\chi_G(x,y)$ is a polynomial (because $\Omega^{\circ}_{\sigma,
\, C(H) \cup T(\sigma)} (x,y)$ is).

\begin{proof}
Let $c: V \longrightarrow [x]$ be a coloring of the mixed graph $G$ that satisfies the coloring conditions~\eqref{cond:bcpomg1} and~\eqref{cond:bcpomg2}.
Note that the colors of the end-points of edges and arcs can be equal only if they are $>y$. 
Let $H$ be a flat of $G$ obtained by contracting all edges and arcs whose end-points have the same color. Thus the vertices in $C(H)$ have color labels $>y$.

Consider $H^{u}$, the underlying undirected graph of the flat $H$. We orient the edges of $H^{u}$ along the color gradient, that is, for the edge $uv$, we introduce the orientation $u \longrightarrow v$ if and only if $c(u) < c(v)$. Let $\sigma$ be such an orientation. 
No two vertices in $H^{u}$ that are connected by an edge have identical color labels. This gives us that $\sigma$ is acyclic. 
Let 
\[
T(\sigma) \ \coloneqq \ \left\{ v \in V(H) : \, \overrightarrow{vw} \in A(H) \ \text{ and }
\ v \longleftarrow w \ \text{ in } \ \sigma \right\} .
\]
As the color gradient is decreasing along the arcs with tail vertices in the set $T(\sigma)$, we have $c(u) >y$ for each $u \in T(\sigma)$ from the coloring constraints. 
Now we regard the acyclic orientation $\sigma$ as a binary relation on the set $V(H^{u})$ defined by $u \preceq v$ if $u \longrightarrow v$. This gives us a bicolored poset $P$ where the vertices in the set $C(H) \cup T(\sigma)$ are celeste elements. The coloring $c$ is an order preserving $(x,y)$\textendash map on $P$. The bivariate order polynomial $\Omega^{\circ}_{\sigma, \, C(H) \cup T(\sigma)} (x,y)$ counts all such order preserving maps. 

Conversely, given a flat $H$ of $G$ and an acyclic orientation $\sigma$ of $H^{u}$,
an order preserving $(x,y)$\textendash map counted by $\Omega^{\circ}_{\sigma, \, C(H) \cup T(\sigma)} (x,y)$ can be extended to a coloring of $G$ as follows. 
All the vertices of $H$ get colors such that the color gradient follows $\sigma$. 
The celeste elements of the bicolored poset induced by the orientation $\sigma$ is given by the set $C(H) \cup T(\sigma)$. Hence the vertices in the set $T(\sigma)$ get colors $>y$. The coloring is then extended to the graph $G$ such that the vertices of the graph $G$ that result in contractions to form the flat $H$ get equal colors $>y$. 
This gives a coloring of the mixed graph $G$.

Consider two distinct colorings $c_1$ and $c_2$ of $G$. We need to show that the corresponding order preserving maps $\phi_1$ and $\phi_2$ are distinct. 

Construct the flats $H_1$ and $H_2$ of the graph by contracting those edges and arcs that
have end-vertices with equal color labels with respect to the colorings $c_1$ and $c_2$ respectively.
If $H_1 \neq H_2$, then the posets on the vertices of the underlying undirected graphs
$H_1^{u}$ and $H_2^{u}$ will be different for each coloring. This will give us distinct
order preserving  $(x,y)$-maps.

Suppose $H_1 = H_2$, that is, both flats are identical, then the underlying undirected graphs $H_1^{u}$ and $H_2^{u}$ will also be identical. Let $\sigma_1$ and $\sigma_2$ be acyclic orientations of $H_1^{u}$ and $H_2^{u}$, respectively. 

Now define $T_i (\sigma_i) \coloneqq \left\{ v \in V(H_i) \mid \overrightarrow{vw} \in
A(H_i) \text{ and } v \longleftarrow w  \text{ in } \sigma_i \text{ of } H_i^{u}  \right\}$
for $i=1,2$. If these sets are distinct, then the celeste elements in the corresponding
bicolored posets will be distinct, resulting in different vertex orderings which will give distinct order preserving $(x,y)$\textendash maps for corresponding colorings.

If for the vertex sets, $T_1(\sigma_1) = T_2(\sigma_2)$ but the acyclic orientations $\sigma_1$ and $\sigma_2$ are distinct, then the posets induced by these acyclic orientations will be distinct resulting in distinct order preserving $(x,y)$\textendash maps for corresponding  graph colorings.

If the flats, the acyclic orientation and the celeste sets are identical, then the
bicolored posets corresponding to both colorings are the same. The bivariate order
polynomial  $\Omega^{\circ}_{\sigma, \, C(H) \cup T(\sigma)} (x,y)$ counts all possible
order preserving $(x,y)$\textendash maps on this bicolored poset exactly once. 
\end{proof}
\begin{figure}
\centering
\includegraphics[width=1.05\textwidth]{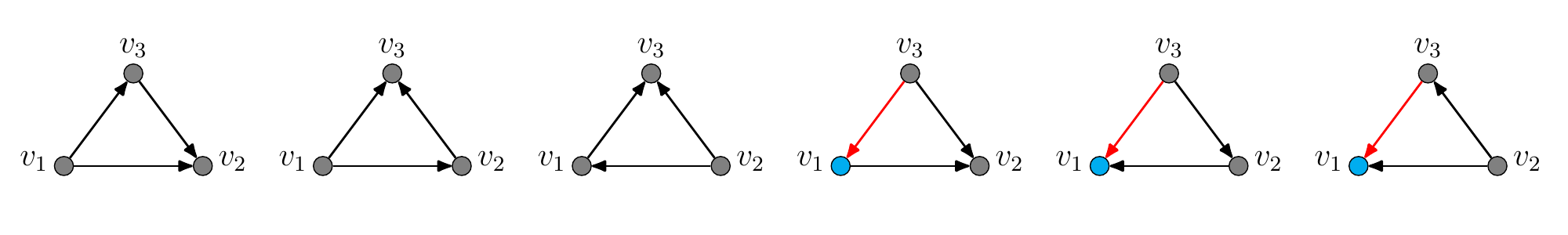}
\includegraphics[width=1.1\textwidth]{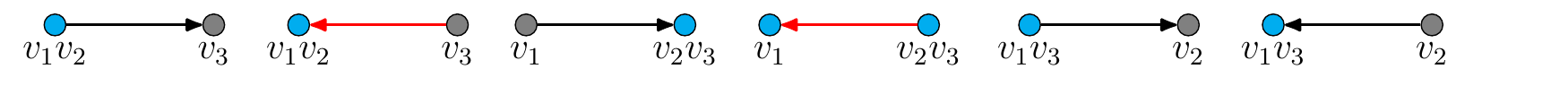}
\includegraphics[width=0.075\textwidth]{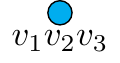}
\caption{Acyclic orientations of contractions of $G$. }
\label{fig:eg2}
\end{figure}

Naturally, an undirected graph is a special case of the above with $A = \emptyset$, and
Theorem~\ref{mainthm2} specializes to one of the results of~\cite{Beckbop}:

\begin{cor}\label{cor:undirpoly}
 For an undirected graph $G=\left(V,E\right)$,
\begin{align*}
\chi_G(x,y)  &=  \sum_{H \text{ \rmfamily flat of } G \hspace{3pt}} 
 \sum_{ \substack{ \sigma \text{ \rmfamily acyclic} \\ \text{\rmfamily
orientation of } H } }\!\!\!\! \Omega^{\circ}_{\sigma, \, C(H)} (x,y) \, .
\end{align*}
\end{cor}
\begin{exam}\label{ex:k3}
For the mixed graph $G$ shown in Figure~\ref{fig:eg2fig1}, our proof of Theorem~\ref{mainthm2} is illustrated by Figure~\ref{fig:eg2}. 

For $H=G$, there are six acyclic orientations of $H^{u}$ as shown in Figure~\ref{fig:eg2}. 

There are three flats obtained by contracting one edge or one arc in $G$. For each underlying undirected graph of a flat, there are two acyclic orientations each. Figure~\ref{fig:eg2} also shows these orientations. There is one flat obtained by contracting two edges or an edge and an arc of the graph resulting in a vertex $v_1v_2v_3$. 

Computing the bivariate order polynomial for each of these orientations yields
\begin{align*}
\chi_G(x,y) \ &= \ 3 {x \choose 3} + 2(x-y) {y \choose 2} + (3y+6) {x-y \choose 2} + 3 {x-y \choose 3} + (x-y)(3y+1) \\
&= \ x^3 - \frac{1}{2}xy^2 - \frac{5}{2}xy+y^2+y \, .
\end{align*}
\end{exam}
%

\section{Reciprocity} \label{sec:bcpmgReciprocity}

An orientation $\sigma$ and a coloring $c: V \longrightarrow [x]$ of the mixed graph $G$
satisfying~\eqref{cond:bcpomg1} and~\eqref{cond:bcpomg2}
%
 are \Def{compatible} if $c(u) \leq c(v)$ for any edge/arc directed from $u$ to $v$ in~$\sigma$.

We define $m_{H}(x,y)$ to be the number of compatible pairs $(\sigma,c)$ consisting of an
acyclic orientation $\sigma$ of $H^u$ and a coloring $c$ with $c(v) > y$ if $v \in C(H) \cup T(\sigma)$. 

\begin{thm}\label{thm:chirec}
For a mixed graph $G$, 
\begin{align*}
\chi_G(-x,-y) \ &= \sum_{H \text{ \rmfamily flat of } G \hspace{3pt}}  (-1)^{\vert V(H) \vert}
\, m_{H} (x,y) \, .
\end{align*}
\end{thm}

\begin{proof} By the reciprocity result of bivariate order polynomials~\eqref{eq:reciprocityBOP}, 
\begin{align*}
\chi_G(-x,-y) \ &= \sum_{H \text{ \rmfamily flat of } G \hspace{3pt}}  \sum_{ \substack{\sigma \text{ \rmfamily acyclic} \\ \text{ \rmfamily orientation of } H^u \\ 
 }}  \!\!\!\! 
 (-1)^{\vert V(H) \vert} \, \Omega_{\sigma, \, C(H) \cup T(\sigma)} (x,y+1)\\
&= \sum_{H \text{ \rmfamily flat of } G \hspace{3pt}}  (-1)^{\vert V(H) \vert} \, m_{H}
(x,y) \, .
\end{align*}
The last equation holds because $\Omega_{\sigma, \, C(H) \cup T(\sigma)} (x,y+1)$ counts the number of order preserving maps $\varphi : \sigma \longrightarrow [x]$ subject to the following conditions:
\begin{itemize}
\item[$\bullet$] for  $u \in C(H) \cup T(\sigma) $, we have $\varphi(u) \geq y+1$;
\item[$\bullet$] the map $\varphi$ is compatible with $\sigma$. \qedhere
\end{itemize}
\end{proof}

Once more, undirected graphs are mixed graphs with $A = \emptyset$, and so
Theorem~\ref{thm:chirec} specializes to one of the main results of~\cite{Beckbop}:
\begin{cor}
 For an undirected graph $G=\left(V,E\right)$, 
\[
\chi_G(-x,-y) \ = \sum_{H \text{ \rmfamily flat of } G} (-1)^{\vert V(H) \vert} \, m_H(x,y) \, ,
\]
where $m_H(x,y)$ is the number of pairs $(\sigma,c)$ consisting of an acylic orientation $\sigma$ of $H$ and a compatible coloring $c: V(H) \longrightarrow [x]$ such that $c(v) >y$ if $v \in C(H)$.
\end{cor}


\acknowledgements
\label{sec:ack}
We are grateful to two anonymous referees for helpful comments. SK thanks Sophia Elia and Sophie Rehberg for encouraging discussions.

\bibliographystyle{abbrvnat}
\bibliography{bib}

\def\cprime{$'$} \def\cprime{$'$}
\begin{thebibliography}{17}
\providecommand{\natexlab}[1]{#1}
\providecommand{\url}[1]{\texttt{#1}}
\expandafter\ifx\csname urlstyle\endcsname\relax
  \providecommand{\doi}[1]{doi: #1}\else
  \providecommand{\doi}{doi: \begingroup \urlstyle{rm}\Url}\fi

\bibitem[Ardila et~al.(2011)Ardila, Bliem, and Salazar]{ardilabliemsalazar}
F.~Ardila, T.~Bliem, and D.~Salazar.
\newblock Gelfand-{T}setlin polytopes and
  {F}eigin-{F}ourier-{L}ittelmann-{V}inberg polytopes as marked poset
  polytopes.
\newblock \emph{J. Combin. Theory Ser. A}, 118\penalty0 (8):\penalty0
  2454--2462, 2011.
\newblock ISSN 0097-3165.

\bibitem[Averbouch et~al.(2008)Averbouch, Godlin, and
  Makowsky]{averbouchgodlinmakowskyelimination}
I.~Averbouch, B.~Godlin, and J.~A. Makowsky.
\newblock A most general edge elimination polynomial.
\newblock In \emph{Graph-theoretic concepts in computer science. 34th
  international workshop, WG 2008, Durham, UK, June 30--July 2, 2008. Revised
  papers}, pages 31--42. Berlin: Springer, 2008.
\newblock ISBN 978-3-540-92247-6.
\newblock \doi{10.1007/978-3-540-92248-3_4}.
\newblock URL \url{citeseerx.ist.psu.edu/viewdoc/summary?doi=10.1.1.158.8649}.

\bibitem[Averbouch et~al.(2010)Averbouch, Godlin, and
  Makowsky]{averbouchgodlinmakowsky}
I.~Averbouch, B.~Godlin, and J.~A. Makowsky.
\newblock An extension of the bivariate chromatic polynomial.
\newblock \emph{European J. Combin.}, 31\penalty0 (1):\penalty0 1--17, 2010.
\newblock ISSN 0195-6698.

\bibitem[Beck et~al.(2012)Beck, Bogart, and Pham]{Golomb}
M.~Beck, T.~Bogart, and T.~Pham.
\newblock Enumeration of {G}olomb rulers and acyclic orientations of mixed
  graphs.
\newblock \emph{Electron. J. Combin.}, 19\penalty0 (3):\penalty0 Paper 42, 13
  pp., 2012.
\newblock ISSN 1077-8926.

\bibitem[Beck et~al.(2015)Beck, Blado, Crawford, Jean-Louis, and
  Young]{weakmixed}
M.~Beck, D.~Blado, J.~Crawford, T.~Jean-Louis, and M.~Young.
\newblock On weak chromatic polynomials of mixed graphs.
\newblock \emph{Graphs Combin.}, 31\penalty0 (1):\penalty0 91--98, 2015.
\newblock ISSN 0911-0119.
\newblock {\tt arXiv:1210.4634}.

\bibitem[Beck et~al.(2020)Beck, Farahmand, Karunaratne, and
  Zuniga~Ruiz]{Beckbop}
M.~Beck, M.~Farahmand, G.~Karunaratne, and S.~Zuniga~Ruiz.
\newblock Bivariate order polynomials.
\newblock \emph{Graphs Combin.}, 36\penalty0 (4):\penalty0 921--931, 2020.
\newblock ISSN 0911-0119.
\newblock \doi{10.1007/s00373-019-02128-w}.
\newblock URL \url{https://doi.org/10.1007/s00373-019-02128-w}.

\bibitem[Dohmen(2015)]{dohmenbivariatepathsandcycles}
K.~Dohmen.
\newblock Closed-form expansions for the universal edge elimination polynomial.
\newblock \emph{Australas. J. Combin.}, 63:\penalty0 196--201, 2015.
\newblock ISSN 1034-4942.

\bibitem[Dohmen et~al.(2003)Dohmen, P{\"o}nitz, and
  Tittmann]{dohmenponitztittmann}
K.~Dohmen, A.~P{\"o}nitz, and P.~Tittmann.
\newblock A new two-variable generalization of the chromatic polynomial.
\newblock \emph{Discrete Math. Theor. Comput. Sci.}, 6\penalty0 (1):\penalty0
  69--89, 2003.
\newblock ISSN 1365-8050.

\bibitem[Furma{\'n}czyk et~al.(2009)Furma{\'n}czyk, Kosowski, Ries, and
  {\.Z}yli{\'n}ski]{furmanczykkosowkiries}
H.~Furma{\'n}czyk, A.~Kosowski, B.~Ries, and P.~{\.Z}yli{\'n}ski.
\newblock Mixed graph edge coloring.
\newblock \emph{Discrete Math.}, 309\penalty0 (12):\penalty0 4027--4036, 2009.
\newblock ISSN 0012-365X.

\bibitem[Hansen et~al.(1997)Hansen, Kuplinsky, and de~Werra]{HaKuWe}
P.~Hansen, J.~Kuplinsky, and D.~de~Werra.
\newblock Mixed graph colorings.
\newblock \emph{Math. Methods Oper. Res.}, 45\penalty0 (1):\penalty0 145--160,
  1997.
\newblock ISSN 1432-2994.

\bibitem[Hillar and Windfeldt(2008/09)]{hillarwindfeldt}
C.~J. Hillar and T.~Windfeldt.
\newblock Fibonacci identities and graph colorings.
\newblock \emph{Fibonacci Quart.}, 46/47\penalty0 (3):\penalty0 220--224,
  2008/09.
\newblock ISSN 0015-0517.
\newblock {\tt arXiv:0805.0992}.

\bibitem[Kotek et~al.(2008)Kotek, Makowsky, and Zilber]{kotekmakowskyzilber}
T.~Kotek, J.~A. Makowsky, and B.~Zilber.
\newblock On counting generalized colorings.
\newblock In \emph{Computer science logic}, volume 5213 of \emph{Lecture Notes
  in Comput. Sci.}, pages 339--353. Springer, Berlin, 2008.

\bibitem[Sotskov and Tanaev(1976)]{sotskovtanaev}
J.~N. Sotskov and V.~S. Tanaev.
\newblock Chromatic polynomial of a mixed graph.
\newblock \emph{Vesc\=\i\ Akad. Navuk BSSR Ser. F\=\i z.-Mat. Navuk}, \penalty0
  (6):\penalty0 20--23, 140, 1976.
\newblock ISSN 0002-3574.

\bibitem[Sotskov et~al.(2002)Sotskov, Tanaev, and Werner]{SoTaWe}
Y.~N. Sotskov, V.~S. Tanaev, and F.~Werner.
\newblock Scheduling problems and mixed graph colorings.
\newblock \emph{Optimization}, 51\penalty0 (3):\penalty0 597--624, 2002.
\newblock ISSN 0233-1934.

\bibitem[Stanley(1970)]{stanleychromaticlike}
R.~P. Stanley.
\newblock A chromatic-like polynomial for ordered sets.
\newblock In \emph{Proc. {S}econd {C}hapel {H}ill {C}onf. on {C}ombinatorial
  {M}athematics and its {A}pplications ({U}niv. {N}orth {C}arolina, {C}hapel
  {H}ill, {N}.{C}., 1970)}, pages 421--427. Univ. North Carolina, Chapel Hill,
  N.C., 1970.

\bibitem[Stanley(1973)]{stanleyacyclic}
R.~P. Stanley.
\newblock Acyclic orientations of graphs.
\newblock \emph{Discrete Math.}, 5:\penalty0 171--178, 1973.

\bibitem[Stanley(2012)]{stanleyec1}
R.~P. Stanley.
\newblock \emph{Enumerative {C}ombinatorics. {V}olume 1}, volume~49 of
  \emph{Cambridge Studies in Advanced Mathematics}.
\newblock Cambridge University Press, Cambridge, second edition, 2012.
\newblock ISBN 978-1-107-60262-5.

\end{thebibliography}
\label{sec:biblio}

\end{document}